\documentclass[12pt]{article}
\usepackage{amsthm}
\usepackage{amsmath,amsfonts,amssymb,rotating,colordvi}
\usepackage{color}
\usepackage{color, colortbl}
\usepackage{xcolor}
\usepackage{array}
\usepackage{multirow}
\usepackage{slashbox}

\usepackage[unicode=true,
 bookmarks=true,bookmarksnumbered=true,bookmarksopen=true,bookmarksopenlevel=2,
 breaklinks=false,pdfborder={0 0 1},backref=false,colorlinks=true]
 {hyperref}

\usepackage{tikz}

\usetikzlibrary{calc}
\usetikzlibrary{shapes.geometric}

\tikzset{
  c/.style={every coordinate/.try}
}

\usetikzlibrary{arrows,shapes,positioning}
\usetikzlibrary{decorations.markings}
\tikzstyle arrowstyle=[scale=1]
\tikzstyle directed=[postaction={decorate,decoration={markings,mark=at position 0.6 with {\arrow[arrowstyle]{stealth};}}}]
\tikzstyle reverse directed=[postaction={decorate,decoration={markings,mark=at position 0.4 with {\arrowreversed[arrowstyle]{stealth};}}}]
\tikzstyle dot=[style={circle,inner sep=1pt,fill}]
\hypersetup{pdftitle={Word-representability of triangulations of grid-covered cylinder graphs},
 pdfauthor={Zongqing Chen, Sergey Kitaev and Brian Y. Sun},
 pdfsubject={ },
 pdfkeywords={ },
 linkcolor=black,citecolor=black,urlcolor=blue,filecolor=blue,pdfpagelayout=OneColumn,pdfnewwindow=true,pdfstartview=XYZ,plainpages=false}

 \newtheorem{thm}{Theorem}[section]

\newtheorem{lem}[thm]{Lemma}

\numberwithin{equation}{section}

\newdimen\Squaresize \Squaresize=11pt
\newdimen\Thickness \Thickness=0.7pt
\def\Square#1{\hbox{\vrule width \Thickness
   \vbox to \Squaresize{\hrule height \Thickness\vss
    \hbox to \Squaresize{\hss#1\hss}
   \vss\hrule height\Thickness}
\unskip\vrule width \Thickness} \kern-\Thickness}

\def\Vsquare#1{\vbox{\Square{$#1$}}\kern-\Thickness}

\def\moins{\raise 1pt\hbox{{$\scriptstyle -$}}}

\begin{document}

\begin{center}
\textbf{\large{On the probability of existence of a universal cycle or a universal word  for a set of words}}
\par\end{center}

\begin{center}
Herman Z.Q. Chen$^{a}$, Sergey Kitaev$^{b}$, Brian Y. Sun$^{c,d}$\\[6pt]
\par\end{center}

\begin{center}
$^{a}$School of Statistics and Data Science, Nankai University,\\
Tianjin 300071,P.R.\ China\\
$^{b}$Department of Computer and Information Sciences,\\
 University of Strathclyde, Glasgow, G1 1XH, UK\\
 $^{c}$ College of Mathematics and System Science,\\
Xinjiang University, Urumqi, Xinjiang 830046, P.R.China\\

$^d$School of Mathematics and Statistics $\&$ KLAS,\\
 Northeast Normal University, \\
 Changchun 130024, Jilin Province, P.R. China

\par\end{center}

\begin{center}
Email: $^{a}$\texttt{zqchern@163.com}, $^{b}$\texttt{sergey.kitaev@cis.strath.ac.uk},
$^{c}$\texttt{brianys1984@126.com}
\par
\end{center}

\par
\noindent \textbf{Abstract.}
A universal cycle, or u-cycle, for a given set of words is a circular word that contains each word from the set exactly once as a contiguous subword. The celebrated de  Bruijn sequences are a particular case of such a u-cycle, where a set in question is the set $A^n$ of all words of length $n$ over a $k$-letter alphabet $A$. A universal word, or u-word, is a linear, i.e. non-circular, version of the notion of a u-cycle, and it is defined similarly.

Removing some words in $A^n$ may, or may not, result in a set of words for which u-cycle, or u-word, exists. The goal of this paper is to study the probability of existence of the universal objects in such a situation. We give lower bounds for the probability in general cases, and also derive explicit answers for the case of removing up to two words in $A^n$, or the case when $k=2$ and $n\leq 4$.


\section{Introduction}

A {\em universal cycle}, or {\em u-cycle},  for a given set $S$ with $\ell$ words of length $n$ is a circular word $u_0u_1\cdots u_{\ell-1}$ that contains each word from $S$ exactly once as a contiguous subword $u_iu_{i+1}\cdots u_{i+n-1}$ for some $0\leq i\leq \ell-1$, where the indices are taken modulo $n$. The notion of a universal cycle was introduced in \cite{CDG92}. The celebrated {\em de  Bruijn sequences} are a particular case of such a u-cycle, where a set in question is the set $A^n$ of all words of length $n$ over a $k$-letter alphabet $A$. Also, a {\em universal word}, or {\em u-word},  for $S$ is a (non-circular) word $u_0u_1\cdots u_{\ell+n-1}$ that contains each word from $S$ exactly once as a contiguous subword $u_iu_{i+1}\cdots u_{i+n-1}$ for some $0\leq i\leq \ell-1$.  In this paper, we assume that $n\geq 2$ and $k\geq 2$ to make all of our definitions be well-defined and to avoid trivialities.

There is a long line of research in the literature dedicated to the study of universal cycles and universal words for various sets of combinatorial structures. For example, see \cite{GG18} and references there in.  We note that the existence of a u-cycle trivially implies the existence of a u-word, but not vice versa. Indeed, if $u_0u_1\cdots u_{\ell-1}$ is a u-cycle for $S$ then $u_0u_1\cdots u_{\ell-1}u_0u_1\cdots u_{n-2}$ is a u-word for $S$. In either case, solving problems on u-cycles and u-word is normally done through considering {\em de  Bruijn graphs}. A de Bruijn graph $B(n,k)$ consists of $k^n$ nodes corresponding to words in $A^n$ and its directed edges are $x_1x_2\cdots x_n\rightarrow x_2\cdots x_nx_{n+1}$ where $x_i\in A$ for $i\in\{1,\ldots,n+1\}$. De Bruijn graphs are an important structure that is used in solving of a variety of problems, e.g. in combinatorics on words~\cite{M2005} and genomics~\cite{ZB2008}.

Let $G=(V,E)$ be a directed graph. A directed path in $G$ is a sequence $v_1,\ldots,v_t$ of distinct nodes such that there is an edge $v_i\rightarrow v_{i+1}$ for each $1\leq i\leq t-1$ and $v_i\in V(G)$. Such a path is  a {\em Hamiltonian path} if it contains all nodes in $G$. A closed Hamiltonian path ($v_t\rightarrow v_1$ is an edge) is a {\em Hamiltonian cycle}. If $G$ has a Hamiltonian cycle then $G$ is {\em Hamiltonian}.  It is well-known, and is not difficult to show, that $B(n,k)$ is Hamiltonian, so any Hamiltonian cycle (resp., path) in $B(n,k)$ corresponds to a u-cycle (resp., u-word) for $A^n$. For example, the cycle $00\rightarrow01\rightarrow11\rightarrow10\rightarrow00$ in $B(2,2)$ corresponds to the u-cycle $0011$, and we can also get a u-word $00110$ from this.\\[-3mm]

\noindent
{\bf The problem in question.} Now, suppose that  we remove $s<k^n$ words from $A^n$ where $A$ is an alphabet of size $k$. The resulting set $S$ may, or may not have a u-cycle or a u-word. Let $P_c(n,k,s)$ and $P_w(n,k,s)$ be the probabilities of the events that $S$ has a u-cycle and u-word, respectively. Then, a natural question is: What are $P_c(n,k,s)$ and $P_w(n,k,s)$? Note that by definition, a u-cycle for $S$ must cover at least $n$ distinct words, and thus if $s>k^n-n$ then $P_c(n,k,s)=0$.

\begin{table}
$$\begin{array}{c|c|c|c}
s & 1 & 2 & 3 \\
\hline
& & &  \\[-4mm]
P_c(2,2,s) & \frac{1}{2} & \frac{1}{6} & 0 \\
& & &  \\[-4mm]
P_w(2,2,s) & 1 & \frac{5}{6} & 1
\end{array}$$
\caption{Values of $P_c(2,2,s)$ and  $P_w(2,2,s)$ for $s\geq 1$}\label{val-Pc(2,2,s)-Pw(2,2,s)}
\end{table}

It is not difficult to see  that if $s=1$, or $s=k^n-1$, then with probability 1 a u-word exists. Indeed, if $s=1$ then removing a word in $A^n$ corresponds to removing a node in $B(n,k)$ that turns a Hamiltonian cycle passing through it to a Hamiltonian path giving a u-word, while if $s=k^n-1$ then only one word remains and it is a u-word. Similarly, it is not difficult to see that if $s=1$ then with probability $1/k^{n-1}$ a u-cycle exists. Indeed, if $s=1$ then  one can only remove words of the form $xx\cdots x$ called loops, and there are $k$ such words, while if $s=k^n-1$ then the only u-cycle of length $n$ can be a loop.

\begin{table}
$$\begin{array}{c|c|c|c|c|c|c|c}
s & 1 & 2 & 3 & 4 & 5 & 6 & 7\\
\hline
& & &  \\[-4mm]
P_c(3,2,s) & \frac{1}{4}& \frac{1}{14} & \frac{1}{28} & \frac{3}{70} & \frac{1}{28} & 0 & 0\\
& & & & & & & \\[-4mm]
P_w(3,2,s) &  1 & \frac{5}{7}  & \frac{13}{28} & \frac{5}{14} & \frac{5}{14}  & \frac{13}{28} & 1 \\
\end{array}$$
\caption{Values of $P_c(3,2,s)$ and  $P_w(3,2,s)$ for $s\geq 1$}\label{val-Pc(3,2,s)-Pw(3,2,s)}
\end{table}

In Tables~\ref{val-Pc(2,2,s)-Pw(2,2,s)}, \ref{val-Pc(3,2,s)-Pw(3,2,s)} and \ref{val-Pc(4,2,s)-Pw(4,2,s)} we present the values of $P_c(n,2,s)$ and $P_w(n,2,s)$ for $n=2,3,4$ obtained by Mathematica 11.3. Even though these tables were obtained by computer, it is possible to check them by hand for $n=2,3$ by considering the existence of a Hamiltonian path in $B(n,2)$. Moreover, in the case of $n=4$, one can consider Eulerian cycles/paths (to be introduced below) in $B(3,2)$ and be also able to check Table~\ref{val-Pc(4,2,s)-Pw(4,2,s)} by hand.\\[-3mm]

\begin{table}
$$\begin{array}{c|c|c|c|c|c|c|c|c|c|c|c|c|c|c|c}
s & 1 & 2 & 3 & 4 & 5 & 6 & 7 & 8 & 9 & 10 & 11 & 12 & 13 & 14 & 15\\
\hline
& & & & & & & & & & & & & & & \\[-4mm]
P_c & \frac{1}{8} & \frac{1}{60} & \frac{1}{140} & \frac{3}{910} & \frac{1}{546} & \frac{1}{728} & \frac{1}{1144} & \frac{1}{1287} & \frac{1}{1430} & \frac{1}{1144} & \frac{1}{728} & \frac{3}{1820} & \frac{1}{280} & 0 & 0 \\
& & & & & & & & & & & & & & & \\[-4mm]
P_w & 1 & \frac{13}{30} & \frac{13}{70} & \frac{1}{10} & \frac{23}{364} & \frac{355}{8008} & \frac{199}{5720} & \frac{62}{2145} & \frac{153}{5720} & \frac{31}{1144} & \frac{3}{91} & \frac{1}{20} & \frac{13}{140} & \frac{29}{120} & 1
\end{array}$$
\caption{Values of  $P_c(4,2,s)$ and $P_w(4,2,s)$ for $s\geq 1$}\label{val-Pc(4,2,s)-Pw(4,2,s)}
\end{table}

\noindent
{\bf Our results in this paper.}  In this paper, we not only provide lower bounds for $P_c(n,k,s)$ and $P_w(n,k,s)$ for any values of $n, k, s$ (summarized in Table~\ref{overview-tab}), but also give exact values in the case of $s=2$ in Theorem~\ref{thm-main}. For example, we will show that
 for $k\geq 3$ and $n\geq 2$,
\[P_w(n,k,2)=\frac{2(2k^n-3k+1)}{k^{n-1}(k^n-1)}.
\]
We remark that some of our proofs require rather subtle considerations, which tend to be more difficult in the case of the binary alphabet.  \\[-3mm]

\begin{table}[h]
\begin{center}
\begin{tabular}{|c|c|c|c|c|}
\hline
 \backslashbox{$k$}{$n$}& 2 & 3 & 4 & $\geq 5$ \\
\hline
2 &  \cellcolor{gray!20!} Table \ref{val-Pc(2,2,s)-Pw(2,2,s)} &  \cellcolor{gray!20!} Table \ref{val-Pc(3,2,s)-Pw(3,2,s)} &  \cellcolor{gray!20!} Table \ref{val-Pc(4,2,s)-Pw(4,2,s)}  & Thm \ref{thm-general-k=2}\\
\hline
$\geq 3$ & Thm \ref{Pcw-2ks-thm} &\multicolumn{3}{c|}{Thm \ref{thm-general}}    \\
\hline

\end{tabular}
\caption{References for the lower bounds for $P_c(n,k,s)$ and $P_w(n,k,s)$. The gray cells refer to exact values.}\label{overview-tab}

\end{center}
\end{table}

\noindent
{\bf Preliminaries.} In this paper, $B'(n,k)$ denotes the graph obtained from $B(n,k)$  after removing $s$ nodes, or $s$ edges depending on the context.

A directed graph is {\em strongly connected} if there exits a directed path from any node to any other node.  A directed graph is {\em connected} if for any pair of nodes $a$ and $b$ there exists a path in the underlying undirected graph. A {\em trail} in a directed graph $G$ is a sequence $v_1,\ldots,v_t$ of nodes such that there is an edge $v_i\rightarrow v_{i+1}$ for each $1\leq i\leq t-1$ and edges are not visited more than once. An {\em Eularian trail} in $G$ is a trail that goes through each edge exactly once. A closed Eulerian trail is an {\em Eulerian cycle}.  A directed graph is {\em Eulerian}  (resp., {\em semi-Eulerian}) if it has an Eulerian cycle (resp., Eulerian trail). Let $d^+(v)$ (resp., $d^-(v)$) denote the out-degree (resp., in-degree)  of a node $v$. The following result is well-known and is not hard to prove.

\begin{thm}\label{Eulerian-trail} A directed graph $G$ is semi-Eulerian if and only if at most one vertex $v$ has $d^+(v) - d^{-}(v) = 1$, at most one vertex $u$ has $d^-(u) - d^{+}(u) = 1$, every other vertex $w$ has $d^{+}(w)=d^{-}(w)$, and $G$ is connected. A graph is Eulerian if and only if it is balanced and (strongly) connected.\end{thm}

The line graph $L(G)$ of a directed graph $G$ is the directed graph whose vertex set corresponds to the edge set of $G$, and $L(G)$ has an edge $e\rightarrow v$  if in $G$, the head of $e$ meets the tail of $v$. It is well-know, and is not difficult to show, that $B(n,k)=L(B(n-1,k))$, and thus a Hamiltonian path (resp., cycle) in $B(n,k)$ corresponds to an Eulerian trail (resp., cycle) in $B(n-1,k)$, and this property will be often used throughout this paper to show the existence of u-cycles and u-words.

Nodes of the form $x^n$, as well as edges of the form $x^n\rightarrow x^n$, are {\em loops}. Nodes of the form $yx^{n-1}$ are  {\em out-special} and nodes of the form $x^{n-1}y$ are {\em in-special}. Out-special and in-special nodes together are {\em special}. The following theorem will be used by us in the paper multiple times.

\begin{thm}[\cite{B95}]\label{thm-Marc} Let $u$ and $v$ be two distinct non-loop nodes in $B(n,k)$. Then, there exist $k$ distinct node-disjoint paths from $u$ to $v$ if and only if $u$ is not out-special and $v$ is not in-special.\end{thm}

\noindent
{\bf Organization of the paper.}  In Sections~\ref{alp-size-3-or-more} and~\ref{binary-alph-sec} we provide the lower bounds for $P_c(n,k,s)$ and $P_w(n,k,s)$ in the cases of $k\geq 3$ and $k=2$, respectively. In Section~\ref{exact-s-2} we give exact values of $P_c(n,k,2)$  and $P_w(n,k,2)$, and in Section~\ref{final-sec} we provide some concluding remarks.

\section{The case of the alphabet of size $k\geq 3$}\label{alp-size-3-or-more}

Let $$S(k,s):=\sum{k\choose s_1}\prod_{i=2}^{n-2}{(i-1){k\choose 2}\choose s_i}$$ where the sum is taken over all $s_1+2s_2+\cdots+(n-2)s_{n-2}=s$ with $s_i\geq 0$ and $1\leq i\leq n-2$. In the next theorem, we will obtain the following lower bounds for $k\geq 3$ and $n\geq 3$:
\begin{equation}\label{gen-cycle}P_c(n,k,s)\geq \frac{S(k,s)}{{k^n\choose s}}\end{equation}
\begin{equation}\label{gen-word} P_w(n,k,s)\geq \frac{1}{{k^n\choose s}}\left(S(k,s) + \sum\alpha{k\choose s_1}\prod_{i=2}^{n-2}{(i-1){k\choose 2}\choose s_i} \right) \end{equation}
where $\alpha=k^n-s+s_1-k+1$ and the sum is taken over all $s_1+2s_2+\cdots+(n-2)s_{n-2}=s-1$ with $s_i\geq 0$ and $1\leq i\leq n-2$. The case of $n=2$ and $k\geq 3$ will be considered in Theorem~\ref{Pcw-2ks-thm} below.

\begin{thm}\label{thm-general} For  $k,n\geq 3$, the lower bounds in \eqref{gen-cycle} and \eqref{gen-word} hold. \end{thm}

\begin{proof}
Assume $k\geq 3$. We observe that removing all $i$-cycles in  $B(n-1,k)$, $1\leq i\leq n-2$, of the binary form, that is, involving only nodes $x_1\cdots x_{n-1}$ for $x_j\in\{x,y\}$ for $1\leq j\leq n$ and $x,y\in\{1,\ldots,k\}$, results in a strongly connected and balanced graph $B'(n-1,k)$. Indeed, clearly  $B'(n-1,k)$ is balanced. To justify that  $B'(n-1,k)$ is strongly connected, we need to show that for any edge $e=A\rightarrow B$ belonging to a removed binary cycle, there is a directed path $P_{AB}$ from $A$ to $B$ which does not go through any other edge from the removed binary cycles. Then, in a path $P_{XY}$ in $B(n-1,k)$ from a node $X$ to a node $Y$, we can replace any such $e$ with $P_{AB}$, so that it gives a path in $B'(n-1,k)$ from $X$ to $Y$.

Suppose $A=x_1\cdots x_{n-1}$ and $B=x_2\cdots x_n$ where all $x_j\in\{x,y\}$ for some $x$ and $y$. Let $z\neq x,y$. Then, $P_{AB}$ is given by
$$A\rightarrow x_2\cdots x_{n-1}z \rightarrow x_3\cdots x_{n-1}zx_2 \rightarrow x_4\cdots x_{n-1}zx_2x_3 \rightarrow\cdots \rightarrow B$$
since no of the edges in $P_{AB}$ belongs to an $i$-cycle for $i<n-1$. So, $B'(n-1,k)$ is Eulerian, and thus its line graph $B'(n,k)$ is Hamiltonian, and there exists a u-cycle corresponding to it.

To justify \eqref{gen-cycle}, we consider $i$-cycles, $2\leq i\leq n-2$, of the form
$$x^{m}y^{j}x^{m}y^{j}\ldots \rightarrow x^{m-1}y^{j}x^{m}y^{j} \ldots  \rightarrow x^{m-2}y^{j}x^{m}y^{j}\ldots \rightarrow \cdots$$
where $x<y$,  $m+j=i$ and $1\leq m,j\leq i-1$. Note that no two of such cycles can share an edge. Thus, we can remove in $B(n-1,k)$ $s_i$ such $i$-cycles  for $1\leq i\leq n-2$ so that the total number of removed edges (corresponding to the total number of removed nodes in $B(n,k)$) is $s$. Clearly, the number of such 1-cycles is $k$, and for $i\geq 2$, the number of such $i$-cycles is $(i-1){k\choose 2}$.

To justify \eqref{gen-word} we note that if all of the $s$ removed edges come from the binary cycles considered above, then the same lower bound as in  \eqref{gen-cycle} will be obtained. This bound, can be improved as follows. Begin with removing $s-1$ edges coming from the binary $i$-cycles as above, which will result in an Eulerian graph, so that we can remove any edge $e$ in such a graph and obtain a semi-Eulerian graph corresponding to a u-word. To count the possibilities to remove such an $e$, we do not want $e$ to be a loop, because this will result in some double counting. However, if $e$ is not a loop, all the cases will be different from already considered cases, because before we were removing entire $i$-cycles for some $i$. This explains the term $\alpha=k^n-(s-1)-(k-s_1)$ in  \eqref{gen-word}.
\end{proof}

In the proof of the next theorem we need the following simple lemma. There, by a circular binary string we mean a number of digits 0 and 1 placed around a circle in positions labeled by $1, 2,\ldots$.

\begin{lem}\label{cycle-no-11-lem} For $k\ge 2$, the number of circular binary strings with $i$ $1$s and $k-i$ $0$s, $0\leq i\leq k$, in which no two $1$s stay next to each other is given by $${k-i-1\choose i-1}+{k-i\choose i}.$$\end{lem}

\begin{proof} Let $h(k,i)$ be the number of binary (non-circular) strings with $i$ $1$s in which no two $1$s stay next to each other. Then, $h(k,i)={k-i+1\choose i}$. Indeed, $h(k,i)$ clearly counts placing $i$ $1$s in a binary string of length $k-i+1$ and then replacing each $1$, but the rightmost $1$, by $10$. For the circular case, if $0$ is in position 1, then we clearly have $h(k-1,i)$ such strings. On the other hand, if $1$ is in position 1 in the circular case, then we have  $h(k-3,i-1)$ such strings since then positions $k$ and $2$ must be occupied by $0$s. This completes the proof.\end{proof}

\begin{thm}\label{Pcw-2ks-thm} Let $n=2$ and $k\geq 3$, $f(k,s):=\sum_{i=0}^{k(k-3)/2}{\frac{k(k-3)}{2}\choose i}{k\choose s-2i}$,
$$g(k,s):=\sum_{i\geq 3}(i-1)!\left({k-i-1\choose i-1}+{k-i\choose i}\right){k\choose s-i},$$ $U(k,s):=f(k,s)+2f(k,s-k)+g(k,s)+2g(k,s-k)$ and $$V(k,s):=k(k-3)\left({k \choose s-1}+2{k\choose s-k-1}\right).$$ Then,
\begin{equation}\label{bound-Pc2ks}
P_c(2,k,s)\geq \frac{U(k,s)}{{k^2\choose s}}.
\end{equation}
Also,
\begin{equation}\label{bound-Pw2ks}
P_w(2,k,s)\geq  \frac{U(k,s) + V(k,s)+2k\sum_{j=1}^{k-1}(f(k,s-j)+g(k,s-j))}{{k^2\choose s}}
\end{equation}
\end{thm}

\begin{proof} Instead of removing $s$ nodes in $B(2,k)$, we consider removing $s$ edges in $B(1,k)$, whose nodes are $k$ loops $1,\ldots k$, and for every pair of nodes $x$ and $y$, both $x\rightarrow y$ and $y\rightarrow x$ are present. We call a 2-cycle in $B(1,k)$ {\em special} if it  involves nodes $x$ and $x+1$ for $1\leq x\leq k-1$, or $1$ and $k$. Clearly, the number of non-special 2-cycles is $\frac{k(k-3)}{2}$.

To justify \eqref{bound-Pc2ks}, note that removing all the edges in any $i$ of non-special 2-cycles in $B(1,k)$, and then removing $s-2i$ loops results in a balanced and strongly connected graph showing the existence of a u-cycle in this case. The number of ways to proceed in this way is clearly given by $f(k,s)$. Moreover, we can proceed in the same way after first removing the $k$ edges either from the cycle $1\rightarrow 2\rightarrow \cdots\rightarrow k\rightarrow 1$, or from the cycle $k\rightarrow (k-1)\rightarrow \cdots\rightarrow k\rightarrow 1$, which explains the term of $2f(k,s-k)$.

To produce a more subtle estimate, we will be removing just a single $i$-cycle for a fixed $i\geq 3$ from $B(1,k)$, which is clearly not counted previously. The only condition on removing such an $i$-cycle is that it must not involved any of the edges in a special 2-cycle for us to guarantee strong connectivity of the obtained graph.  The number of ways to selected $i$ nodes to form such an $i$-cycle is given by Lemma~\ref{cycle-no-11-lem}, and since there are edges in both directions between any pair of selected nodes, there are $(i-1)!$ ways to choose a cycle on the chosen nodes.  The remaining $s-i$ edges to be removed after removing $i$-edges in an $i$-cycle can be chosen among the $k$ loops. This explains the term of $g(k,s)$. Finally, removing the $k$ edges in either the cycle  $1\rightarrow 2\rightarrow \cdots\rightarrow k\rightarrow 1$, or the cycle $k\rightarrow (k-1)\rightarrow \cdots\rightarrow k\rightarrow 1$, and then removing an $i$-cycle as above results in a balanced and strongly connected graph, and explains the term $2g(k,s-k)$. This completes justification of~\eqref{bound-Pc2ks}.

To justify \eqref{bound-Pw2ks}, first note that all cases considered in proving \eqref{bound-Pc2ks} can be used in the case of u-words. To improve the bound, we note that a directed path (on distinct nodes) of length $j$, $1\leq j\leq k-1$,  consisting of edges coming from special 2-cycles, can be removed, and then some other cycles can be removed as discussed in the case of u-cycles, which will result in a semi-Eulerian graph and thus corresponds to a u-word. There are 2 ways to pick the direction of such a path on special 2-cycles, and $k$ ways to pick its start, justifying the term of $2k\sum_{j=1}^{k-1}(f(k,s-j)+g(k,s-j))$. Finally, the following two options also result in semi-Eulerian graph not considered above:
\begin{itemize}
\item remove any non-loop edge among the $k(k-3)$ edges coming not from special 2-cycles, and the remaining edges can be removed from loops. This gives $k(k-3){k\choose s-1}$ possibilities;
\item  remove the $k$ edges in either the cycle  $1\rightarrow 2\rightarrow \cdots\rightarrow k\rightarrow 1$, or the cycle $k\rightarrow (k-1)\rightarrow \cdots\rightarrow k\rightarrow 1$, and then remove one more non-loop edge among the $k(k-3)$ edges coming not from special 2-cycles, and the remaining edges can be removed from loops, which gives $2k(k-3){k\choose s-k-1}$ possibilities.
\end{itemize}
This explains the term of $V(k,s)$ and completes the proof of \eqref{bound-Pw2ks}.
\end{proof}

\section{The case of the alphabet size $k=2$}\label{binary-alph-sec}

Recall that in Tables~\ref{val-Pc(2,2,s)-Pw(2,2,s)}, \ref{val-Pc(3,2,s)-Pw(3,2,s)} and \ref{val-Pc(4,2,s)-Pw(4,2,s)} we present the values of $P_c(n,2,s)$ and $P_w(n,2,s)$ for $n=2,3,4$.

Let $$T(n,s)=\sum{2\choose s_1}{1 \choose s_2}{1 \choose s_n}$$ where the sum is taken over all $s_1+2s_2+(n-1)s_{n-1}=s$ with $0\leq s_1\leq 2$, $0\leq s_2\leq 1$ and $0\leq s_{n-1}\leq 2$.
Then, for $k=2$ and $n\geq 5$, we will show in the next theorem that
\begin{equation}\label{gen-cycle-2}P_c(n,2,s)\geq \frac{T(n,s)}{{2^n\choose s}}.\end{equation}
\begin{equation}\label{gen-word-2}P_w(n,2,s)\geq \frac{T(n,s)+\sum(2^n-s+s_1-1){2\choose s_1}{1 \choose s_2}{1 \choose s_n}}{{2^n\choose s}}\end{equation}
where the sum is taken over all $s_1+2s_2+(n-1)s_{n-1}=s-1$ with $0\leq s_1\leq 2$, $0\leq s_2\leq 1$ and $0\leq s_{n-1}\leq 2$.

\begin{thm}\label{thm-general-k=2} For $n\geq 5$, the lower bounds in \eqref{gen-cycle-2} and \eqref{gen-word-2} hold. \end{thm}

\begin{proof}
We have $k=2$ and $n\geq 5$. We observe that removing two loops, the 2-cycle, and the two $(n-1)$-cycles of the form
$$x^{n-2}y\rightarrow x^{n-3}yx \rightarrow x^{n-4}yx^2\rightarrow\cdots\rightarrow xy^{n-2}\rightarrow x^{n-2}y$$
in  $B(n-1,k)$ results in a strongly connected and balanced graph $B'(n-1,k)$. Indeed, clearly  $B'(n-1,k)$ is balanced. To justify that  $B'(n-1,k)$ is strongly strongly connected, we need to show that for any edge $e=A\rightarrow B$ belonging to a removed cycle, there is a directed path $P_{AB}$ from $A$ to $B$ which does not go through any other edge from the removed  cycles. Then, in a path $P_{XY}$ in $B(n-1,k)$ from a node $X$ to a node $Y$ we can replace any such $e$ with $P_{AB}$ giving a path in $B'(n-1,k)$ from $X$ to $Y$.We consider two cases.

\noindent
{\bf Case 1.} $A=x^iyx^{n-i-2}$ and $B=x^{i-1}yx^{n-i-3}$ for $0\leq i\leq n-2$. If $i=0$ then $P_{AB}=A\rightarrow x^{n-1}\rightarrow B$. If $i=1$ then  $P_{AB}$ is given by
$$A\rightarrow yx^{n-3}y\rightarrow x^{n-3}yy\rightarrow \cdots \rightarrow y^{n-1}\rightarrow y^{n-2}x\rightarrow y^{n-3}x^2\rightarrow\cdots \rightarrow B.$$
If $i=n-2$, then $P_{AB}$ is given by
$$A\rightarrow x^{n-3}yy\rightarrow \cdots \rightarrow y^{n-1}\rightarrow y^{n-2}x\rightarrow y^{n-3}x^2\rightarrow\cdots \rightarrow B.$$
In all other cases, $P_{AB}$ is given by
$$A\rightarrow x^{i-1}yx^{n-i-2}y\rightarrow x^{i-2}yx^{n-i-2}y^2\rightarrow x^{i-3}yx^{n-i-2}y^2x \rightarrow $$
$$x^{i-4}yx^{n-i-2}y^2x^2\rightarrow \cdots \rightarrow y^2x^{i-1}yx^{n-i-4} \rightarrow yx^{i-1}yx^{n-i-3}\rightarrow B.$$

\noindent
{\bf Case 2.} $A=x_1x_2 \cdots xyxy$ and $B= x_2 \cdots yxyx$ are in the 2-cycle where $x\neq y$. Then, for even $n$, $P_{AB}$ is given by
$$A\rightarrow x_2 \ldots yxyy \rightarrow x_3 \ldots xyyx \rightarrow x_4 \ldots yyxy \rightarrow x_5 \ldots yxyx \rightarrow \cdots \rightarrow B,$$
and for odd $n$, $P_{AB}$ is given by
$$A\rightarrow x_2 \ldots yxyy \rightarrow x_3 \ldots xyyx \rightarrow x_4 \ldots yyxx \rightarrow x_5 \ldots yxxy \rightarrow $$
$$x_6 \ldots xxyx \rightarrow x_7 \ldots xyxy \rightarrow \cdots \rightarrow B.$$
So, $B'(n-1,k)$ is Eulerian, and thus its line graph $B'(n,k)$ is Hamiltonian, and there exists a u-cycle corresponding to it.

To justify \eqref{gen-cycle-2}, we note that no two of the two loops, one 2-cycle and two $(n-1)$-cycles considered above can share an edge. Thus, we can remove in $B(n-1,k)$ $s_i$ such $i$-cycles  for $i\in\{1,2,n-1\}$ so that the total number of removed edges (corresponding to the total number of removed nodes in $B(n,k)$) is $s$.

To justify \eqref{gen-word-2} we note that if all of the $s$ removed edges come from the cycles considered above, then the same lower bound as in  \eqref{gen-cycle-2} will be obtained. This bound, can be improved as follows. Begin with removing $s-1$ edges coming from the $i$-cycles as above, which will result in an Eulerian graph, so that we can remove any edge $e$ in such a graph and obtain a semi-Eulerian graph corresponding to a u-word. To count the possibilities to remove such an $e$, we do not want $e$ to be a loop, because this will result in some double counting. However, if $e$ is not a loop, all the cases will be different from already considered cases, because before we were removing entire $i$-cycles for some $i$. This explains the factor of $2^n-(s-1)-(2-s_1)$ in  \eqref{gen-word-2}.
\end{proof}

\section{Exact values of $P_c(n,k,2)$  and $P_w(n,k,2)$}\label{exact-s-2}

\begin{thm}\label{thm-main-cycle} We have $P_c(2,2,2)=\frac{1}{6}$ and for $n\geq 3$ and $k\geq 2$, $$P_c(n,k,2)=\frac{k(k-1)}{{k^n\choose 2}}.$$
\end{thm}

\begin{proof}  If two nodes are removed in $B(2,2)$, the only possibility for the graph to stay Hamiltonian (and thus to correspond to a u-cycle) is if the removed nodes are loops, which explains that $P_c(2,2,2)=\frac{1}{6}$. On the other hand, if $n\geq 3$ and $k\geq 2$ then in order to obtain an Eulerian graph by removing two edges in $B(n-1,k)$ we must either remove two loops, or remove a 2-cycle. Each of these gives ${k \choose 2}$ possibilities thus explaining the formula for  $P_c(n,k,2)$.  \end{proof}

The proof of Theorem~\ref{thm-main} relies on the following theorem, which looks like an intuitively true statement, but its proof is rather involved and requires consideration of many cases, and we were not able to find this result in the literature.

\begin{thm}\label{thm-X} Let $e=a\rightarrow b$ be an edge in $B(n,k)$. Then, there exists a Hamiltonian cycle in $B(n,k)$  that goes through $e$, with the only exception when $k=2$, $a$ is out-special and $b$ is in-special. \end{thm}

\begin{proof}  If $k=2$, $a=yx^{n-1}$ and $b=x^{n-1}y$ then no Hamiltonian cycle can cover the loop $x^n$ and go through $e$  because the only edge coming to $x^n$ comes from $a$. This is not the case for $k\geq 3$.

If either $a$ or $b$ is a loop $x^n$, then the statement is true. Indeed, if $k=2$ then there is only one edge coming in to  $x^n$, and one edge coming out of  $x^n$, so these edges will be part of any Hamiltonian cycle. On the other hand, if $k\geq 3$, then suppose $a=yx^{n-1}$ and $b=x^n$; the case when $a=x^n$ can be considered similarly.  Since $B(n,k)$ has a Hamiltonian cycle and the corresponding u-cycle $U$, it will go through an edge $zx^{n-1}\rightarrow x^{n}$, $x\neq z$. If $y=z$ we are done. Otherwise, we can swap all $y$'s and $z$'s in $U$ to obtain the desired Hamiltonian cycle from the new u-cycle.

Thus, we can assume that neither $a$ nor $b$ is a loop.

In what follows, we will use the following approach. We will be considering edges $e_1=A\rightarrow B$ and $e_2=B\rightarrow C$ in $B(n-1,k)$ corresponding to $a$ and $b$, respectively. Next, we will demonstrate that after removing $e_1$ and $e_2$ (corresponding to removing $a$ and $b$ in $B(n,k)$) the obtained graph $B'(n-1,k)$ remains connected. This is done via finding alternative directed paths $P_{XY}$ from $X$ to $Y$, $X\neq Y$, where $X,Y\in\{A,B\}$ or $X,Y\in\{B,C\}$. Together with the fact that $B'(n-1,k)$ is balanced if $A=C$, or otherwise  $A$ has one extra edge coming in, and $C$ has one extra edge coming out,  $B'(n-1,k)$ has an Eulerian trail corresponding to a Hamiltonian path in $B'(n,k)$ obtained from $B(n,k)$ after removing $a$ and $b$. Such a Hamiltonian path can clearly be extended to a Hamiltonian cycle in $B(n,k)$ by adding back the removed edge $e$.

Suppose that $n=3$. If $k=2$ then we have 8 possibilities for $e$ (loops are not involved, and we cannot have $a$ be out-special and $b$ be in-special). Each of the 8 possible choices of $e$ can be found in one of  the following two Hamiltonian cycles in $B(3,2)$ giving the desired result:
$$100\rightarrow  000\rightarrow 001\rightarrow010\rightarrow101\rightarrow011\rightarrow111\rightarrow110\rightarrow100$$
$$001\rightarrow  011\rightarrow 111\rightarrow110\rightarrow101\rightarrow010\rightarrow100\rightarrow000\rightarrow001$$
Thus we can assume $k\geq 3$. Let $a=xyz$ and $b=yzh$. Then $A=xy$, $B=yz$ and $C=zh$. Letting $t\neq x,z$ we see that $P_{AB}=A\rightarrow yt\rightarrow ty\rightarrow B$ if $ty\neq yz$, or else
$P_{AB}=A\rightarrow yt\rightarrow B$. $P_{BC}$ is found in the same way.

In what follows, we assume that $n\geq 4$ (and $a$ and $b$ are not loops).

If $A=C$ then $A, B, C$ form a 2-cycle, and for $n\geq 3$ none of these vertices is special. Thus, $P_{AB}$ and $P_{BC}$ exist by Theorem~\ref{thm-Marc}. So, we can assume that $A\neq C$.

Suppose that one of $A$, $B$, or $C$ is a loop.

\noindent
{\bf Case 1.} $A=x^{n-1}$ is a loop, $B=x^{n-2}y$, $C=x^{n-3}yz$, and $y\neq x$. $P_{BC}$ exist by Theorem~\ref{thm-Marc}, and for $k\geq 3$, $P_{AB}$ is given by
$$A=x^{n-1}\rightarrow x^{n-2}t\rightarrow x^{n-3}tx\rightarrow x^{n-4}tx^2\rightarrow\cdots \rightarrow tx^{n-2} \rightarrow B=x^{n-2}y$$
where $t\neq y,x$. For $k=2$, we note that $P_{AB}$ does not exist in this case. However, it is sufficient for us to prove that there exists $P_{BA}$ that does not use the edge $e_2$ ($e_1$ clearly will not be used). Such a path is given by
$$B=x^{n-2}y\rightarrow x^{n-3}y\bar{z} \rightarrow x^{n-4}y\bar{z}x\rightarrow x^{n-4}y\bar{z}x^2\rightarrow\cdots \rightarrow A=x^{n-1}$$
where $\bar{z}$ denotes the letter distinct from $z$. In the case $z=x$ the path above has one extra step than otherwise.

\noindent
{\bf Case 2.} $B=y^{n-1}$ is a loop, $A=xy^{n-2}$, $C=y^{n-2}z$, $y\neq x$, and $z\neq y$. For $k\geq 3$, $P_{BC}$ is essentially $P_{AB}$ in Case 1, and $P_{AB}$ is given by
$$A=xy^{n-2}\rightarrow y^{n-2}t\rightarrow y^{n-3}ty\rightarrow y^{n-4}ty^2\rightarrow\cdots \rightarrow ty^{n-2} \rightarrow B=y^{n-1}$$
where $t\neq y,x$. The case $k=2$ corresponds to $a$  being out-special, and $b$ being in-special, and it is the exception in the statement of the theorem (the loop $B$ becomes non-reachable from any other node).

\noindent
{\bf Case 3.} $C=z^{n-1}$ is a loop, $A=xyz^{n-3}$, $B=yz^{n-2}$, and $y\neq z$. $P_{AB}$ exist by Theorem~\ref{thm-Marc}, and for $k\geq 3$ $P_{BC}$ is given by
$$B=yz^{n-2}\rightarrow z^{n-2}t\rightarrow z^{n-3}tz\rightarrow z^{n-4}tz^2\rightarrow\cdots \rightarrow tz^{n-2} \rightarrow C=z^{n-1}$$
where $t\neq y,x$. Similarly to Case 1, for $k=2$, $P_{BC}$ does not exist, but we can find $P_{CB}$ not using $e_1$ and $e_2$:
$$C=z^{n-1}\rightarrow z^{n-2}\bar{x}\rightarrow z^{n-3}\bar{x}y\rightarrow z^{n-4}\bar{x}yz\rightarrow z^{n-5}\bar{x}yz^2\rightarrow\cdots\rightarrow B=yz^{n-2}$$
where $\bar{x}$ is the letter different from $x$ so the last step is not the edge $e_1$.

Thus, we can assume that none of $A$, $B$, or $C$ is a loop. Moreover, we only need to consider the following two cases, because otherwise, $P_{AB}$ and $P_{BC}$ are given by  Theorem~\ref{thm-Marc}.

\noindent
{\bf Case i.} $A=yx^{n-2}$ is out-special and $B=x^{n-2}z$ is in-special, where $x\neq y,z$.  In this case, $P_{AB}=yx^{n-2}\rightarrow  x^{n-1} \rightarrow B=x^{n-2}z$.

\noindent
{\bf Case ii.} $B=yx^{n-2}$ is out-special and $C=x^{n-2}z$ is in-special, where $x\neq y,z$. This is essentially Case i.
\end{proof}

\begin{thm}\label{thm-main}
We have $P_w(2,2,2)=\frac{5}{6}$ and for $n\geq 3$,
\[P_w(n,2,2)=\frac{2^{n}-3}{(2^n-1)2^{n-3}}
\]
Moreover, for $k\geq 3$ and $n\geq 2$, we have
\[P_w(n,k,2)=\frac{2(2k^n-3k+1)}{k^{n-1}(k^n-1)}.
\]
\end{thm}

\begin{proof} If $n=2$ and $k=2$, then it is easy to see that the graph obtained from $B(2,2)$ by removing two nodes has a Hamiltonian path unless the removed nodes are 01 and 10. This gives $P_w(2,2,2)=\frac{5}{6}$.

So, we can assume that either $n\geq 3$ and $k\geq 2$ or $n\geq 2$ and $k\geq 3$.

Let $a$ and $b$ be the nodes in $B(n,k)$ corresponding to the removed words of length $n$ over $k$-letter alphabet.
Note that $B'(n-1,k)$ is semi-Eulerian (in particular, respecting the conditions on in-degrees and out-degrees) if and only if
\begin{itemize}
\item either $a$ or $b$ is a loop, in which case clearly a Hamiltonian path in $B'(n,k)$ exists, or
\item $a\rightarrow b$ (or $b\rightarrow a$) is an edge in $B(n,k)$, in which case a Hamiltonian path in $B'(n,k)$ exists by Theorem~\ref{thm-X} with one exception.
\end{itemize}
Thus, exactly one of the following four cases of choosing $a$ can occur in order for $B'(n,k)$ to have a Hamiltonian path.

\noindent
{\bf Case 1.} $a$ is a loop, in which case $b$ can be any node. Clearly, there are $k(k^n-1)$ ways to choose such $a$ and $b$.

\noindent
{\bf Case 2.} (i) $a=xyxy\cdots x$ or (ii) $a=xyxy\cdots y$. $b$ can only be of the form, in case (i) $tt\cdots t$ or $yxyx\cdots y$ or $zxyxy\cdots y$ or $yxyx\cdots xz$ where $z\neq y$, and in case (ii)  $tt\cdots t$ or $yxyx\cdots y$ or $zxyxy\cdots x$ or $yxyx\cdots yz$ where $z\neq y$. In either case, $b$ can be chosen in $k + 1 + (k-1) + (k-1)=3k-1$  ways  depending on the respective choices of $t$ and $z$. There are $k(k-1)$ choices to pick $a$ giving in total for this case $k(k-1)(3k-1)$ possibilities.

\noindent
{\bf Case 3.} $a=yx^{n-1}$ is out-special or  $a=x^{n-1}y$ is in-special. In either case, $b$ can be either a loop, or the other end point of an edge coming into $a$ or going out of $a$. There are $k$ loops, $k$ edges coming in, and $k$ edges coming out of $a$, but one of these edges is connected to a loop. Thus, in this case we have $3k-1$ choices for $b$, and in total $2k(k-1)(3k-1)$ possibilities (2 corresponds to the choices of being out-special or in-special). However, the last formular only works for $k\geq 3$, because when $k=2$, we cannot remove $a$ and $b$ connected by an edge, when both of them are special (in this case a loop becomes isolated). So, if $k=2$ we have $2k(k-1)(3k-1-1)=16$ possibilities.

\noindent
{\bf Case 4.} In all other cases of $a$, $b$ can be any of $2k$ nodes connected to $a$ by an edge, or any of $k$ loops. So, we have $$3k(k^n-k-k(k-1)-2k(k-1))=3k(k^n-3k^2+2k)$$ possibilities.

Since every pair $(a,b)$ appears twice in our arguments, Cases 1--4 give
$$P_w(n,2,2)=\frac{2(2^n-1)+10+16+6(2^n-8)}{(2^n-1)2^n}$$
and for $k\geq 3$, $P_w(n,k,2)=$
$$\frac{k(k^n-1)+k(k-1)(3k-1)+2k(k-1)(3k-1)+3k(k^n-3k^2+2k)}{(k^n-1)k^n}.$$
\end{proof}

\section{Directions of further research}\label{final-sec}
A universal cycle (or a universal word) for an arbitrary set $S$ is a cyclic sequence (or a non-circular sequence) whose substrings of length $n$ encode $|S|$ distinct instances in $S$. U-cycles and u-words have been studied for a wide variety of combinatorial objects including permutations~\cite{jrj, kitaev}, partitions~\cite{ca}, subsets~\cite{subset}, multisets~\cite{hjz}, labeled graphs~\cite{lg}, various functions and passwords; for more information, the reader is referred to ~\cite{wong}. There are much study on universal cycles or universal words because of their applications including dynamic connections in overlay networks~\cite{net}, genomics~\cite{ZB2008}, software calculation of the ruler function in computer words~\cite{wong}, etc. An interesting direction of research would be extending our studies of $P_c(n,k,s)$ and $P_w(n,k,s)$ to other combinatorial structures. Also, it would be interesting to explore new methods to compute the exact values of $P_c(n,k,s)$ and $P_w(n,k,s)$. Finally, our bounds for the general case presented in \eqref{gen-cycle} and \eqref{gen-word} can be improved by conducting a more subtle analysis of the removed cycles in the proof of Theorem~\ref{thm-general}, and exploring how far the improvement could go is another interesting direction.

\section*{Acknowledgements}
The first author was supported by the Fundamental Research Funds for the Central Universities, Nankai University (Grant Nnumber 63191349).
The third author was partially supported by the National Natural Science Foundation of China (Grant Numbers 11701491, 11726629 and 11726630) and the China Postdoctoral Science Foundation (Grant Number 2017M621188).

\end{document}